\documentclass[12pt]{amsart}

\usepackage{mathrsfs}
\usepackage{latexsym}
\usepackage{amssymb}
\usepackage{amsmath}
\usepackage{amsfonts}
\usepackage{amsthm}
\usepackage{multirow}
\usepackage{xcolor}
\usepackage{verbatim}
\usepackage{caption}
\usepackage{graphicx}
\usepackage{hyperref}
\usepackage{enumerate}
\usepackage{booktabs}

\def\a{\alpha} \def\b{\beta}   
\def\t{\tau}  \def\o{\omega} 
\def\Ga{\Gamma}

\def\ZZ{\mathbb{Z}}    

 \def\lcm{{\rm lcm}} \def\char{{\sf \,char\,}}
  
\def\le{\leqslant} \def\leq{\leqslant} \def\ge{\geqslant} 
 
\def\l{\langle} \def\r{\rangle}

  \def\Cos{{\sf Cos}}

\def\calA{\mathcal{A}}                       
                \def\calF{\mathcal{F}}    
          \def\calH{\mathcal{H}}             
                 \def\calL{\mathcal{L}}    
\def\calM{\mathcal{M}}                     

\def\Aut{{\mathrm{Aut}}}

 \def\Hol{{\rm Hol}}  \def\Sym{{\rm Sym}}

\def\A{\mathrm{A}}    
\def\C{\mathrm{C}}    
\def\D{\mathrm{D}}    
\def\K{\mathrm{K}}    
\def\N{\mathrm{N}}    
\def\O{\mathrm{O}}    
\def\S{\mathrm{S}}    
\def\Z{\mathrm{Z}}    
\def\RM{\mathsf{RotaMap}}
\def\BiRM{\mathsf{BiRotaMap}}
\def\RevMap{\mathsf{RevMap}}
\def\BiRevMap{\mathsf{BiRevMap}}
\def\lv{\lvert}
\def\rv{\rvert}
\newtheorem{theorem}{Theorem}[section]
\newtheorem{lemma}[theorem]{Lemma}

\newtheorem{definition}[theorem]{Definition}
\newtheorem{example}[theorem]{Example}

\def\qed{{\hfill$\Box$\bigskip}\medbreak}

\begin{document}
	
	\title[]{Arc-transitive maps with edge number\\  coprime to the Euler characteristic -- I}
\thanks{This work was partially supported  by NSFC grant 11931005}

\author[Li]{Cai Heng Li}
\address{SUSTech International Center for Mathematics\\
Department of Mathematics\\
Southern University of Science and Technology\\
Shenzhen, Guangdong 518055\\
P. R. China}
\email{lich@sustech.edu.cn}

\author[Liu]{Luyi Liu}
\address{ShenZhen International Center for Mathematics\\
Southern University of Science and Technology\\
Shenzhen, Guangdong 518055\\
P. R. China}
\email{12031108@mail.sustech.edu.cn}

	\maketitle
	
\begin{abstract}
This is one of a series of papers which aim towards a classification of edge-transitive maps of which the Euler characteristic and the edge number are coprime.
This one establishes a framework and carries out the classification work for arc-transitive maps with solvable automorphism groups, which illustrates how the edge number impacts on the Euler characteristic for maps.
The classification is involved with the constructions of various new and interesting arc-regular maps.
	
\textit{Key words:} arc-regular, maps, Euler characteristic
		
\end{abstract}

\section{Introduction}
A {\it map} is a $2$-cell embedding of a graph into a closed surface.  
Throughout the paper, we denote by $\calM = (V, E, F)$ a map with vertex set $V$, edge set $E$, and face set $F$.  
The {\it underlying graph} $(V, E)$ of $\calM$ is denoted by $\Gamma$, and the supporting surface of $\calM$ is denoted by $\mathcal{S}$.  
The {\it Euler characteristic} of a map $\calM = (V, E, F)$ is  
\[
\chi(\calM) = |V| - |E| + |F|.
\]  
Clearly, the number of edges $|E|$ directly affects $\chi(\calM)$.  
In this paper, we investigate highly symmetric maps where $\gcd(|E|,\ \chi(\calM)) = 1$.
Moreover, we always assume that $\calM$ has no loops.

Let $\calM = (V, E, F)$ be a map.
An {\it arc} of $\calM$ is an ordered adjacent pair $(\a, e)$ of a vertex $v$ and an edge $e$, while a {\it flag} of $\calM$ is an adjacent triple $(\a, e, f)$ of a vertex $\a$, an edge $e$, and a face $f$.  
Let $e$ be an edge with endpoints $\a$ and $\b$.
By the definition of a map, there are exactly two faces $f$ and $f'$ adjacent to $e$. 
The edge $e$ is adjacent to two arcs $(\a, e)$ and $(\b, e)$, and four flags:  
$(\a, e, f)$, $(\a, e, f')$, $(\b, e, f)$, and $(\b, e, f')$.
The arc set and flag set of $\calM$ are denoted by $\calA$ and $\calF$, respectively, and satisfy $|\calF| = 2|\calA| = 4|E|$.  
We denote by $E(v)$ and $E(f)$ the sets of edges incident to vertex $\a$ and face $f$, respectively.  
The {\it vertex valency} of $\a$ is defined as $|E(\a)|$, and the {\it face length} of $f$ is defined as $|E(f)|$.

An {\it automorphism} of $\calM$ is a permutation $\sigma$ of its flags that preserves all incidence relations.  
The set of all automorphisms of $\calM$ forms a group, denoted by $\Aut(\calM)$.  
If an automorphism $\sigma \in \Aut(\calM)$ fixes a flag, then it must fix all flags, and therefore $\sigma$ is the identity element.  
This implies that $\Aut(\calM)$ acts semiregularly on the flag set $\calF$ of $\calM$.  
If $\Aut(\calM)$ acts transitively on $\calF$, then $\calM$ is called a {\it regular map}.  
If $\Aut(\calM)$ acts transitively on the arc set $\calA$ but intransitively on $\calF$, then it acts regularly on $\calA$, and $\calM$ is termed an {\it arc-regular map}.  
Similarly, $\calM$ is called {\it edge-regular} if $\Aut(\calM)$ acts regularly on the edge set $E$.  
The classification problem of certain arc-transitive maps has been extensively studied; see \cite{few-VF, birotarynp, regularnp, RotaMap, RevMap, howsymmetry} and references therein.

A map is called {\it locally finite} if all vertex valencies and face lengths are finite.  
Locally finite edge-transitive maps and their automorphism groups are classified by Graver and Watkins~\cite{etm14tp} into fourteen types based on the local actions of their automorphism groups.  
Among these types, five are arc-transitive which are explain as follows.  

Let $G \leq \Aut(\calM)$ act arc-transitively on $\calM$.  
Then $G$ is transitive on the vertex set $V$, and the valency $\ell = |E(\a)|$ is constant for all $\a \in V$, called the {\it valency} of $\calM$.  
Furthermore, there exists an involution $z \in G$ that swaps the two arcs $(\a, e)$ and $(\b, e)$ adjacent to an edge $e \in E$.

First, suppose that $G$ acts transitively on the flag set of $\calM$, then $G = \Aut(\calM)$, and $\calM$ is a regular map.  
Furthermore, there exist involutions $x, y \in G$ such that  
\[
G = \l x, y, z \r, \quad G_v = \l x, y \r = \D_{2\ell}, \quad G_e = \l x, z \r = \D_4.
\]  
The triple $(x, y, z)$ is called a {\it regular triple} of $G$.  

Now assume that $G$ is not transitive on the flag set.
Then $G$ is regular on the arc set of $\calM$, and $\calM$ is a  $G$-arc-regular map.
In this case, $G_v$ is cyclic or dihedral of order equal to $\ell$.
If $G_\a=\l a\r=\ZZ_\ell$, then $\calM$ is said to be {\it $G$-vertex-rotary}, and $(a,z)$ is called a {\it rotary pair} for $G$.
On the other hand, if $G_\a=\l x,y\r=\D_\ell$, then $\calM$ is said to be {\it $G$-vertex-reversing}, and $(x,y,z)$ is called a {\it reversing triple} for $G$.
Further, according to the action of $z$ on $\{f,f'\}$, vertex-rotary maps are divided into {\it rotary maps} and {\it bi-rotary maps}, and vertex-reversing maps are divided into {\it reversing maps} and {\it bi-reversing maps}, as shown in the table below.

\begin{table}[h]
	\caption{Five types of arc-transitive maps}\label{tab}
	\centering
	\begin{tabular}{cccc}
		\toprule[0.01in]
		Type & $G_\a$ & $(f,f')^z$ &Name\\ 
		\midrule[0.01in]
		$1$ & $\l x,y\r=\D_{2k}$ & $(f,f')\text{ or }(f',f)$&regular\\ 
		$2^*$ & $\l x,y\r=\D_k$ & $(f,f')$&reversing\\
		$2^{{\rm P}}$ &$\l x,y\r=\D_k$ & $(f',f)$ &bi-reversing\\ 
		$2^*$ex & $\l a\r=\ZZ_k$ & $(f',f)$ & rotary (orientably regular)\\
		$2^{{\rm P}}\rm{ex}$ & $\l a\r=\ZZ_k $& $(f,f')$& bi-rotary\\
		\bottomrule[0.01in]
	\end{tabular}
	
\end{table}

We classify arc-transitive maps $\calM=(V,E,F)$ with $\gcd(|E|,\chi(\calM))=1$ in two papers.
This one presents a classification with $\Aut(\calM)$ solvable, and the second one~\cite{non-solvable} completes the work for the case with $\Aut(\calM)$ nonsolvable.

Noticing that arc-transitive embeddings of multicycles are determined in \cite{GuoLiu}, and arc-transitive embeddings of complete multigraphs are classified in~\cite{RevMap}.
For the orientably-regular embeddings, see~\cite{Siran-complete}.
The main result of this paper is stated as follows.

\begin{theorem}\label{thm:main}
Let $\calM=(V,E,F)$ be a map with $\gcd(\chi(\calM),|E|)=1$.
Assume that $G\leq\Aut(\calM)$ is indecomposable, solvable and arc-transitive on $\calM$, and the underlying graph of $\calM$ is neither $\K_4^{(\lambda)}$ nor $\C_n^{(\lambda)}$ with $n\ge2$.
Then one of the following holds:
\begin{enumerate}[{\rm(1)}]

\item $G=\l g\r{:}\l h\r$, where $\C_{\l h\r}(g)\leqslant\Phi(\l h\r)$, $|h|=2m$ and $g^{h^m}=g^{-1}$; the map $\calM$ is $G$-vertex-rotary with a rotary pair $(a,z)$ is equivalent to $(gh^i,h^m)$ or $(gh^{2i},h^m)$, where $\gcd(i,|h|)=1$;
the number of inequivalent rotary pairs for $G$ is 
$\frac{2\varphi(|h|)}{|\tau_{h_0}|}$ if $G_2\cong\D_{2^e}$ 
with $e \geqslant 3$;
and $\varphi(|h|)/|\tau_{h_0}|$ otherwise, where $\l h_0 \r = \C_{\l h \r}(g)$. 
\vskip0.1in

\item[] {\rm In the following, Sylow $2$-subgroups of $G$ are dihedral.
Let $u$, $v$ be two involutions such that $\l u,v\r=G_2$ and $ w=uv$. \vskip0.1in}

\item $G=(\l g\r\times\l g_v\r){:}(\l u,v\r\times\l h\r)$, where $|g|,|g_v|$ and $|h|$ are pairwise coprime, $\C_{\l h\r}(gg_v)\leqslant \Phi(\l h\r)$, and $\l u, v\r=\D_{2^e}$ with $e\geqslant 2$ such that \[(g,g_v)^{v}=(g,g_v^{-1}) \text{ and } (g,g_v)^{u}=(g,g_v),\]
where  $g_v\neq 1$; and either
\begin{enumerate}[{\rm(i)}] 
	\item $\calM$ is $G$-vertex-reversing, $G=\l g_v\r{:}\l u,v\r$, relabeling if necessary, a reversing triple $(x,y,z)$ is equivalent to 
	\[(u,g_vw^{2j}v,v)\text{ , or }(v,g_vw,u)\text{ when } |w|=2,\] and there are $6\cdot 2^{e-2}$ inequivalent reversing triples of $G$; or\vskip0.1in
	
	\item $\calM$ is $G$-vertex-rotary, and a rotary pair $(a,z)$ is equivalent to  
	\[(gg_vwh^j,v) \text{ or }(gg_vvh^j,v),\]
	where $\gcd(j,|h|)=1$, and there are $2\varphi(|h|)/|\t_{h_0}|$ inequivalent rotary pairs of $G$.
\end{enumerate}
\vskip0.1in

\item  $G=(\l g_u\r\times\l g_v\r\times\l g\r){:}(\l u,v\r\times\l h\r)$, where $|g_u|,|g_v|,|g|$ and $|h|$ are pairwise coprime, $\C_{\l h\r}(g_ug_vg)\leqslant \Phi(\l h\r)$, and $\l u, v\r=\D_{2^e}$ with $e\geqslant 2$ such that
\[\mbox{$(g_u,g_v, g)^v=(g_u,g_v^{-1},g^{-1})$ and $(g_u,g_v,g)^u=(g_u^{-1},g_v,g^{-1})$,}\]
where if $g_u=1$, then $g_v\not=1$ and $g\not=1$;
suppose that $g_v=1$, then either  
\begin{enumerate}
	\item[{\rm (1)}] $\calM$ is $G$-vertex-reversing, $h=1$, and relabeling if necessary, a reversing triple $(x,y,z)$ is equivalent to 
	\[(g_u g^{k_1} w^{2i} u,\, g^{k_2} v,\, u),\]
	where $\gcd(k_1, k_2) = 1$, $\ell_1$ is odd, and there are
	\[
	6\cdot 2^{e-3} \cdot \frac{|\{(k_1, k_2) \mid \gcd(k_1, k_2) = 1\}|}{\varphi(|g|)}
	\]
	inequivalent reversing triples of $G$;\vskip0.1in
	
	\item[{\rm (2)}] $\calM$ is $G$-vertex-rotary, and a rotary pair $(a,z)$ is equivalent to
	\[(g_u g w h^\ell, u)\text{ or }(g_u g w^{2} v h^\ell, u),\]
	and there are $2 \varphi(|h|)/|\tau_{h_0}|$ inequivalent rotary pairs in $G$.
\end{enumerate}

\vskip0.1in

\item $G=(\l g\r\times\l w^2,v\r){:}((\l h\r{:}\l u\r)\times\l c\r)$, where $\l u,v\r=\D_8$ such that $h^u=h^{-1}$ and $g^{u}=g^{-1}$, $\l h\r$ is a Sylow $3$-subgroup such that $(w^2v,v,w^2)^h=(v,w^2,w^2v)$ and $[g,h]=1$ and  $\l g\r$ and $\l c\r$ are Hall subgroups such that $\C_{\l c\r}(g)\leqslant\Phi(\l c\r)$;
and either

\begin{enumerate}[{\rm(i)}]
	
\item $\calM$ is $G$-vertex-rotary with a rotary pair $(a,z)$ being equivalent to 
\[(g t h^j c^k, u)\text{ or }(g t h^j u c^k, u),\] 
where  $\gcd(j,3)=\gcd(k,|c|)=1$, and $t\in \{v, w^2\}$,
and there are 
\[\varphi(|h|) \varphi(|c|) / 2 |\t_{h^3}| |\t_{c_0}|\]
inequivalent rotary pairs of $G$.

\item $\calM$ is $G$-vertex-reversing, $c=1$;
let $t\, t'\in\l w^2,v\r$,  relabeling if necessary, a reversing triple $(x,y,z)$ is equivalent to one of the following triples:

\begin{enumerate}
	\item $(t (gh)^i u, t', u)$ with $t'\neq 1$,\vskip0.1in
	\item $(t (gh)^i u, g^{i'} h^{3j_0} w^2u, u)$,\vskip0.1in
	\item $(t (gh)^i u, g^{i'} t' h^{j'} u, u)$ with $(t, t')\neq (1, 1)$,
\end{enumerate}
\vskip0.1in
where $\gcd(i,|gh|)=1$.

\end{enumerate} 

\end{enumerate}

\end{theorem}

\section{Preliminaries}

In this section, we collect some basic definitions and preliminary results which will be used in the ensuing sections.

Let $\calM=(V,E,F)$ be a map, and $G\le\Aut(\calM)$ be arc-regular on $\calM$.
Then the stabilizer $G_e$ of an edge $e$ is generated by an involution $z$ say, explicitly,
\[G_e=\l z\r \cong \ZZ_2,\]
and $(\a,e)^z=(\b,e)$.
Let $f,f'\in F$ be incident with $e$.
Then $z$ either fixes $f$ and $f'$, or interchanges $f$ and $f'$, so that
\[z:\ (f,f')\mapsto(f,f')\ \mbox{or}\ (f',f).\]

Since $G$ acts regularly on the arc set, the vertex stabilizer $G_\a$ is regular on $E(\a)$.  
The pair $(G_\a, z)$ defines an arc-regular coset graph $\Gamma = (V, E)$, where  
\[
V = [G : G_\a] = \{ G_\a g \mid g \in G \}, \quad 
E = [G : \langle z \rangle] = \{ \langle z \rangle h \mid h \in G \}
\]  
with incidence defined by non-empty intersection: a vertex $G_\alpha g$ and an edge $\langle z \rangle h$ are incident if and only if $G_\alpha g \cap \langle z \rangle h \neq \emptyset$.  
This coset graph, denoted by $\Cos(G, G_\alpha, \langle z \rangle)$, may have multiple edges; see~\cite{RotaMap} for details.

Assume that $G_\a=\l a\r$.
Then the rotary pair $(a,z)$ determines two vertex-rotary embeddings of $\Ga$, namely, a rotary map or a bi-rotary map, defined below.

\begin{definition}\label{cons-rotary}
Let $G$ be a group generated by $a,z$ with $z$ being an involution, and
let $\Ga=\Cos(G,\l a\r,\l z\r)$.
Let $V=[G:G_\a]$, $E=[G:\l z\r]$, and let
\begin{itemize}
\item[(i)] $\RM(G,a,z)$ be the embedding of $\Ga$ with face set $[G:\l az\r]$, and
\item[(ii)] $\BiRM(G,a,z)$ be the embedding of $\Ga$ with face set $[G:\l z,z^a\r]$,
\end{itemize}
with incidences defined by non-empty intersection.
Then $\RM(G,a,z)$ is a rotary map, and $\BiRM(G,a,z)$ is a bi-rotary map.
\end{definition}

Now let $G_\a=\l x,y\r$ with $x$ and $y$ being involutions.
Then the reversing triple $(x,y,z)$ determines two vertex-reversing embeddings of $\Ga$, namely, a reversing map or a bi-reversing map, defined below.

\begin{definition}\label{cons-reversing}
Let $G$ be a group generated by involutions $x,y,z$, and
let $\Ga=\Cos(G,\l x,y\r,\l z\r)$.
Let
\begin{itemize}
\item[(i)] $\RevMap(G,x,y,z)$ be the embedding of $\Ga$ with $F=[G{:}\l x,z\r]\cup[G{:}\l y,z\r]$,
\item[(ii)] $\BiRevMap(G,x,y,z)$ be the embedding of $\Ga$ with $F=[G{:}\l x,y^z\r]$,
\end{itemize}
with incidence defined by non-empty intersection.
Then $\RevMap(G,x,y,z)$ is a reversing map, and $\BiRevMap(G,x,y,z)$ is a bi-reversing map.
\end{definition}

To end of this section, we provide some definitions for symbols used in this paper.
For an integer $n$, let $n_p$ be the $p$-part, which means that $n=n_pm$ such that $n_p$ is a $p$-power and $\gcd(p,m)=1$.
Let $\pi(n)$ be the set of prime divisors of the integer $n$.

For a group $G$, let $\pi(G)=\pi(|G|)$, the set of prime divisors of the order $|G|$.
For a subset $\pi$ of $\pi(G)$, denote by $G_\pi$ a Hall $\pi$-subgroup, and $G_{\pi'}$ a Hall $\pi'$-subgroup of $G$, where $\pi'=\pi(G)\setminus\pi$.
If $\pi=\{p\}$, then $G_\pi=G_p$ is a Sylow $p$-subgroup of $G$.
For a subgroup $H\leqslant G$, let $\N_G(H)$ and $\C_G(H)$ be the normaliser and the centralizer of $H$ in $G$, respectively.
Let $\Phi(G)$ be the Frattini subgroup of $G$, which consists of all non-generators of $G$.
By $\ZZ_n$, we mean a cyclic group of order $n$, and $\D_{2n}$ is a dihedral group of order $2n$.

\section{The edge number and the Euler characteristic}\label{Sec:|E|}

In this section, we investigate the effect of the Euler characteristic on edge numbers and automorphism groups of maps.
Also, we provided a characterization of automorphism groups of maps such that the Euler characteristic and the edge number are coprime.

\begin{lemma}\label{lem:E-chi}
Assume that $\calM$ has constant vertex valency $\ell$ and constant face length $k$, and $\gcd(\chi(\calM),|E|)=1$.
Then $|E|$ divides $k\ell$, $\lv V \rv$ divides $2k$ and $|F|$ divides $2\ell$.
\end{lemma}

\begin{proof}
Since each edge is adjacent to two vertices and two faces, there are
\[
|V| \cdot \ell = 2|E|\text{ and }|F| \cdot k = 2|E|.
\]   
Combining these with the Euler characteristic formula $\chi(\calM) = |V| - |E| + |F|$, we substitute  
\[
|V| = \frac{2|E|}{\ell}, \quad |F| = \frac{2|E|}{k} 
\]  
to obtain  
\[
k\ell \chi(\calM) = (2\ell - k\ell + 2k)|E|.
\]  
It follows from $\gcd(\chi(\calM), |E|) = 1$ that $|E|$ divides $k\ell$. 
Consequently, $|V| = \frac{2|E|}{\ell}$ divides $\frac{2k\ell}{\ell} = 2k$, and $|F| = \frac{2|E|}{k}$ divides $\frac{2k\ell}{k} = 2\ell$.
\end{proof}
		
This lemma shows that, if $\gcd(\chi,|E|)=1$ then the edge number $|E|$ is small relative to the valency and the face length.

The next lemma characterizes Sylow subgroups of the automorphism group of a map whose Euler characteristic is coprime to the number of edges.
	
\begin{lemma}\label{(chi,|A|)=1}
	Let $\calM=(V,E,F)$ be a map, and $G$ be a subgroup of $\Aut(\calM)$.
	Assume that $\gcd(\chi(\calM),|E|)=1$.
	Then the following statements are true:
	\begin{enumerate}[{\rm(1)}]
		\item $\gcd(\chi(\calM),|G|)$ divides $4$, and $\gcd(\chi(\calM),|G|)\neq 1$ only if the edge number $|E|$ is odd;\vskip0.1in
		\item each Sylow subgroup of $G$ is a cyclic or dihedral subgroup;\vskip0.1in
		\item $|G|=\lcm\{|G_\o|: \o\in V\cup E\cup F\}$.
	\end{enumerate}
\end{lemma}

\begin{proof}

Recall that $\Aut(\calM)$ acts semiregular on the flag set $\calF$.
So  $|G|$ divides $|\calF|$.
Let $\chi=\chi(\calM)$ for convenience.
Since  $|\calF|=4|E|$ and $\gcd(\chi,|E|)=1$, $\gcd(\chi,|G|)$ divides $\gcd(\chi,|\calF|)=\gcd(\chi,4|E|)=\gcd(\chi,4)$.
Suppose that $\gcd(\chi,|G|)\neq 1$.
Then $\chi$ is even, and so the edge number $|E|$ is odd, as in part~(1).

Let $p$ be a prime divisor of $|G|$.  
First, suppose that $p \mid \chi$.  
It follows from statement~(1) that $p = 2$ and $|E|$ is odd.
Since $|G|$ divides $4|E|$, the Sylow $2$-subgroups of $G$ have order dividing $4$.  
Explicitly, Sylow $2$-subgroups are cyclic or dihedral. 

Now, suppose that $p \nmid \chi$.  
We claim that exists $\omega \in V \cup E \cup F$ such that $p \nmid \frac{|G|}{|G_\omega|}$.  
Otherwise, $p$ divides the length of each $G$-orbit on $V\cup E \cup F$, and so $p$ divides $|V|-|E|+|F|=\chi$ which is the sum all $G$-orbits length.
Contradiction comes to $p\nmid \gcd(\chi,|G|)$, and so the claim is proved. 
Therefore Sylow $p$-subgroups of $G$ are isomorphic to a subgroup of $G_\o$ by Sylow's theorem.
So they are cyclic or dihedral as $G_\omega$ is cyclic or dihedral for each $\omega\in V \cup E \cup F$, as in statement~(2).

Let $\ell=\lcm\left( |G_\omega| \mid \omega \in V \cup E \cup F \right)$ and $d=\gcd\left(|G|/|G_\o| \mid \o\in V\cup E\cup F\right)$.
Then $|G|=\ell d$, and the greatest common divisor $d$ divides $\gcd(\chi,|G|)$.

Assume that $d\neq 1$.
By statement~(1), $\gcd(\chi,|G|)$ is $2$ or $4$, and $|G|=2|E|$ or $4|E|$ where the edge number $|E|$ is odd.
There follows two cases as below.
\begin{enumerate}
\item[(a)] If $|G|=4|E|$, then $\calM$ is a regular map which implies that edge stabilizers in $G$ are isomorphic to $\D_4$.
Hence $4$ divides $\ell$, and so $d=1$ which is a contradiction.\vskip0.1in

\item[(b)] If $|G|=2|E|$, then $d=2$ and $\ell=|E|$ is odd.
Therefore  $ |G_\o| $ is odd for each $\o \in V \cup E \cup F$.
A contradiction comes to that $|E|$ is the sum of all lengths of $G$-orbits on $E$, which are all even. 
\end{enumerate}
Above all, $d=1$ and so $G=\ell$ as in statement~(3). 
\end{proof}

Regular maps on orientable surface of genus $p+1$, where $p$ is prime, have been classified\cite{RegOri-chi-2p}.
In particular, automorphism groups of regular maps $\calM$ with order coprime to $\chi/2$ are determined.
Here by part~(1) of Lemma~\ref{(chi,|A|)=1}, if a $G$-arc-transitive map $\calM$ satisfies $\gcd(\chi,|E|)=1$ and $\gcd(\chi/2,|G|)=1$,  the edge number $|E|$ of $\calM$ is odd.
To avoid overlaps, we will not discuss on the above case.

\section{Almost Sylow-cyclic groups}\label{sec:G}

A group is called an {\it almost Sylow-cyclic group} if its odd order Sylow subgroups are cyclic and each Sylow $2$-subgroup contains an index $2$ cyclic subgroup, refer to \cite{Reg-chi}.
Non-solvable almost Sylow-cyclic groups were completely classified by Suzuki \cite{Suzuki-1955} and Wong~\cite{Wong-1966}, and solvable almost Sylow-cyclic groups were characterized by Zassenhaus \cite{Zassenhaus}.

The authors of this paper presents an explicitly classification of solvable almost Sylow-cyclic groups, extending Theorem~4.1 of \cite{Reg-chi}.

\begin{theorem}[\cite{A-ASG}]\label{lem:G}

	Let $G$ be an indecomposable solvable finite group of which each Sylow subgroup is cyclic or dihedral.
	Then $G$ is one of the following five types:
	\begin{enumerate}[{\rm(I)}]
		\item $G=(\l g\r{:}\l h\r)$, where  $\C_{\l h\r}(g)\leqslant\Phi(\l h\r)$.
		\vskip0.1in
		
		\item[]  {\rm In the following, Sylow $2$-subgroups of $G$ are dihedral.
		Let $u$, $v$ be two involutions such that $\l u,v\r=G_2$ and $ w=uv$.} \vskip0.1in
		
		\item $G=(\l g\r\times\l g_v\r){:}(\l u,v\r\times\l h\r)$, where $|g|,|g_v|$ and $|h|$ are pairwise coprime, $\C_{\l h\r}(gg_v)\leqslant \Phi(\l h\r)$, and $\l u, v\r=\D_{2^e}$ with $e\geqslant 2$ such that \[(g,g_v)^{v}=(g,g_v^{-1}) \text{and} (g,g_v)^{u}=(g,g_v),\]
		where $g_v\neq 1$;\vskip0.1in
		
		\item $G=(\l g_u\r\times\l g_v\r\times\l g\r){:}(\l u,v\r\times\l h\r)$, where $|g_u|,|g_v|,|g|$ and $|h|$ are pairwise coprime, $\C_{\l h\r}(g_ug_vg)\leqslant \Phi(\l h\r)$, and $\l u, v\r=\D_{2^e}$ with $e\geqslant 2$ such that
		\[\mbox{$(g_u,g_v, g)^v=(g_u,g_v^{-1},g^{-1})$ and $(g_u,g_v,g)^u=(g_u^{-1},g_v,g^{-1})$,}\]
		where if $g_u=1$, then $g_v\not=1$ and $g\not=1$;\vskip0.1in
		
		\item $G=(\l g\r\times\l u,v\r){:}\l h\r$, where $|g|$, $|h|$ are coprime,  $|h|$ is divisible by $3$, $\C_{\l h^3\r}(g)\leqslant\Phi(\l h^3\r)$, and $\l u,v\r\cong \D_4$ such that  $(u,v,w)^h=(v,w,u)$;\vskip0.1in
		
		\item $G=(\l g\r\times\l w^2v,v\r){:}((\l h\r{:}\l u\r)\times\l c\r)$, where $\l h\r$ is a Sylow $3$-subgroup,  $|g|$, $|h|$, $|c|$ are pairwise coprime, $\C_{\l c\r}(g)\leqslant\Phi(\l c\r)$, and  $\l u,v\r=\D_8$  such that  \[ (w^2v,v,w)^{h}=(v,w,w^2v) \text{ and } h^u=h^{-1}.\]
		
	\end{enumerate}
\end{theorem}

Automorphism groups of cyclic groups and dihedral groups are well known.
For convenience, we provide the definitions below. 
\begin{definition}\label{def:tau-l0}
	{\rm
		For a cyclic group $L=\l l\r$ and an element $l_0\in\Phi(L)$,
		define an automorphism  of $L$:
		\[\t_{l_0}:\ l\to l^{m{|l|\over|l_0|}+1},\]
		where $m$ is an integer such that $\gcd(m|l_0|+1,|l|)=1$.
		Then $\l\t_{l_0}\r=\Aut(L)_{(L/\l l_0\r)}$. 
		\qed
	}
\end{definition}

\begin{definition}\label{Auto-Dihedral}
	{\rm
		Let $Q=\l w\r{:}\l u\r=\D_{2m}$ with $ m \geqslant 3$.
		Define automorphisms of $Q$ \begin{enumerate}
		\item $\xi:\ (w,w^iu)\mapsto(w,w^{i+1}u)$, $i=1,...,m$, and\vskip0.1in
		
		\item $\mu:\ wu^j \mapsto w^iu^j$, where $i$ is a primitive divisor of $m-1$ and $j=1,2$. 
		
		\end{enumerate} 
		
		Then $\Aut(Q)=\l \xi\r{:}\l \mu\r\cong \Hol(\ZZ_m)$.
		\qed
	}
\end{definition}

Automorphism groups of such group are determined, as quoted below.
\begin{theorem}\label{Thm:AutoG}
	Let $G$ be an indecomposable solvable finite group of which each Sylow subgroup is cyclic or dihedral.
	Then $G=H{:}L$ where $H$, $L$ are Hall subgroups, and 
	\[\Aut(G)=\widetilde{H}{:}\calH\times \calL,\]
	where $\Aut(H)\cong \calH\leqslant\Aut(G)_{(L)}$ and $\Aut(L)_{(L/C)}\cong \calL\leqslant\Aut(G)_{L}\cap \Aut(G)_{(H)}$.
	
	Following notations in Theorem~\ref{lem:G}, one of the following statements holds:
	\begin{enumerate}[{\rm (I)}]
    \item $H=\l g_{2'}\r$, $L=\l g_{2}\r{:}\l h\r$, and $\calL\cong \Hol(\l g_2\r)\times \l \t_{h_0}\r$;\vskip0.1in
    \item $H=\l gg_v\r$, $L=\l u,v\r\times \l h\r$, and $\calL\cong \l \xi^2\r{:}\l \mu\r \times \l \t_{h_0}\r$;\vskip0.1in
    \item $H=\l g_ug_vg\r$, $L=\l u,v\r\times \l h\r$, and $\calL\cong \l \xi^2\r{:}\l \mu\r\times \l \t_{h_0}\r$;\vskip0.1in
    \item $H=\l g\r$, $L=\l u,v\r{:}\l h\r$, and $\calL$ is isomorphic to
    \[
    \begin{aligned}
    &(\l \tilde{u},\tilde{v}\r \times \l \t_{h^3}\r){:} (\l \tilde{h}\r{:}\l \mu_r\r) \text{,if  } [h,g]=1;\\
    &\l \tilde{u},\tilde{v}\r {:} \l \tilde{h}\r \times \l \t_{h^3}\r\text{,if }[h^3,g]=1.
    \end{aligned}\]\vskip0.1in
    \item $H=\l g\r$, $L=\l w^2,v\r{:}(\l h\r{:}\l u\r)\times \l c\r$, and $\calL$ is isomorphic to
    \[
    \begin{aligned}
    	&(\l \tilde{u},\tilde{v}\r \times \l \t_{h^3}\r){:} (\l \tilde{h}\r{:}\l \mu_r\r) \times \l \t_{c_0}\r \text{,if  } [h,g]=1;\\
    	&\l \tilde{u},\tilde{v}\r {:} \l \tilde{h}\r \times \l \t_{h^3}\r \times \l \t_{c_0}\r \text{,if }[h^3,g]=1,
    \end{aligned}\]
    where $\l c_0\r=\C_{\l c\r}(g)$.

	\end{enumerate}
\end{theorem}

\section{Type I: metacyclic groups}\label{sec:metac}

In this section, we always assume that $G$ is metacyclic, that is, $G$ is a group of type I.

\begin{lemma}\label{lem:G-dih}
Assume that $G=\l g\r{:}\l h\r=\D_{2m}$ is a dihedral group, and $\calM$ is a $G$-arc-regular map with underlying graph $\Gamma$.

Then either
\begin{enumerate}[{\rm(1)}]
	\item $\calM$ is $G$-vertex-rotary with rotary pair $(a,z)$, where the pair is equivalent to $(g,g^jh)$ or $(h, gh)$, and $\Gamma$ is  $\K_2^{(m)} $ or $\C_m $, respectively; or \vskip0.1in

	\item  $\calM$ is $G$-vertex-reversing with reversing triple $(x,y,z)$, where the triple is equivalent to  $\{h, g h, g^i h\}$ or $\{h, g h, g^{m/2}\}$, and $\Gamma$ is a multiloop or $\C_{m/2}^{(2)}$.
\end{enumerate}

\end{lemma}

\begin{proof}
Suppose that $\calM$ is $G$-vertex-rotary with a rotary pair $(a,z)$.  
Then $G = \l a,z \r$, and so $a\neq 1$.
It follows that $a = g^i$ or $a = g^i h$ for some integer $i$.  
Consequently, $(a,z)$ is of the form $(g^i, g^j h)$ with $\gcd(i, m) = 1$, or $(h, g^j h)$ with $\gcd(j, m) = 1$.  
By equivalence under $\Aut(G)$, the pair $(a,z)$ is either $(g, g^j h)$ or $(h, gh)$.  

The subgroup $\l a \r$ corresponds to a vertex stabilizer in $G$, and its order equals the vertex valency of $\Gamma$.  
When $|a| = m$, this implies $\Gamma \cong \K_2^{(m)}$; if $|a| = 2$, then $\Gamma \cong \C_m$.  
This aligns with case (1).

Now assume $\calM$ is $G$-vertex-reversing with a reversing triple $(x,y,z)$.  
In the dihedral group $G$, any two non-commuting involutions generate $G$; let these be $x$ and $y$.  
Without loss of generality, we may set $(x,y) = (h, gh)$.  
The third involution $z$ can be chosen arbitrarily.  

After possible relabeling, the triple $\{x,y,z\}$ is equivalent to:  
\begin{itemize}
	\item $\{h, gh, g^i h\}$ for some $i$, yielding a multiloop graph $\Gamma$, or  
	\item $\{h, gh, g^{m/2}\}$ when $m$ is even, resulting in $\Gamma \cong \C_{m/2}^{(2)}$.
\end{itemize}

Up to duality, vertex stabilizers in $G$ are isomorphic to either $G$ or $\D_4$, determining the graph structure as above.
\end{proof}

We now consider non-dihedral groups $G$ and characterize necessary conditions for the existence of rotary pairs in $G$.
\begin{lemma}\label{lem:G-metac}
	Let $\calM$ be a $G$-arc-transitive map. 
	Assume that $G = \l g \r{:}\l h \r$, where $\C_{\l h \r}(g) \leqslant \Phi(\l h \r)$ and $|h| \geqslant 3$.
	Then the following statements hold:
	\begin{enumerate}[{\rm (1)}]
		\item The map $\calM$ is $G$-vertex-rotary. \vskip0.1in
		
		\item The group satisfies $G = \l g \r{:}\l h \r$ 
		with $|h| = 2m$ and $g^{h^m} = g^{-1}$. \vskip0.1in
		
		\item A rotary pair $(a,z)$ is equivalent to $(gh^i, h^m)$ 
		or $(gh^{2i}, h^m)$, where $\gcd(i,|h|)=1$. 
		The latter case occurs only if $G_2 \cong \D_{2^e}$ 
		with $e \geqslant 3$. \vskip0.1in
		
		\item Let $\l h_0 \r = \C_{\l h \r}(g)$. 
		The number of inequivalent rotary pairs for $G$ is 
		$\frac{2\varphi(|h|)}{|\tau_{h_0}|}$ if $G_2\cong\D_{2^e}$ 
		with $e \geqslant 3$, 
		and $\frac{\varphi(|h|)}{|\tau_{h_0}|}$ otherwise.
	\end{enumerate}
\end{lemma}

\begin{proof}
	
The group $G$ has no reversing triples.  
Suppose otherwise, then $G/\l g \r \cong \l h \r$ with $|h| \geqslant 3$ would be generated by involutions, leading to a contradiction.  
This establishes part~(1).  

It remains to classify the rotary pairs of $G$.  
Assume first that $|h|$ is odd.  
Since $\l g \r{:}\l h \r$ is indecomposable and non-cyclic, $|g|$ is also odd.  
Consequently, there exist no involutions in $G$.

Hence, $\l h \r$ have even order $2m$, containing a unique involution $u = h^m$.  
If $u$  centralizes a Sylow $p$-subgroup $\l g \r_p$ of $\l g \r$, then $\l g \r_{p'}{:}\l u \r$ would be normal in $G$, making the quotient $G/(\l g \r_{p'}{:}\l u \r)$ cyclic.  
This would split $G$ as $\l g \r_p \times G_{p'}$, contradicting the assumption.  
Therefore, $u$ centralizes no Sylow subgroup of $\l g \r$, and $g^u = g^{-1}$ follows as in part~(2).  

We now classify rotary pairs of $G$.  
A involution $z$ is equivalent to $u$ under $\Aut(G)$, so set $z = u = h^m$.  
Write $a = g^i h^j$ for integers $i, j$.  
From $G = \l a, z \r = \l g^i h^j, u \r \leq \l g^i, h^j, u \r$, we deduce $\gcd(i, |g|)=1$ and $\gcd(j, |h|) \mid 2$.  
Note $\gcd(j, |h|) = 2$ requires $|h|_2 = 2$.  
Theorem~\ref{lem:G} further implies $a$ is equivalent to $gh^j$ or $gh^{2j}$ where $\gcd(j, |h|)=1$.  

Still, we need to check that $\l a,z\r=G$.  
It follows from $g^u=g^{-1}$ that $[z,a]=a^z a^{-1}=g^{-2}$, and thus  
\[
\l a,z\r = \l g_{2'}, a_{2}, a_{\pi(|h|)_{2'}}, z\r.
\]  
Suppose that $|h_2|\geqslant 4$.  
Then $g_2=1$, so $\l a_2\r=\l h_2\r$ is a Sylow $2$-subgroup of $G$.  
On the other hand, $\l a_2,z\r$ is also a Sylow $2$-subgroup of $G$ when $|h_2|=2$, since $h_2=u$ and $\l a_2,z\r=\l g_2,u\r$.  
In either case, $\l a,z\r=G$ holds, proving part~(3).  

Next, we show part~(4):  
There are $|g|$ involutions in $G \setminus \Z(G)$.  
By part~(3), there are $\varphi(|g|)\varphi(|h|)$ rotary pairs $(a,z)$ for each $z \in G \setminus \Z(G)$.  
Since $\l a,z\r=G$, $\Aut(G)$ acts semiregularly on the set of rotary pairs of $G$.  
By Theorem~\ref{lem:G}, $|\Aut(G)|=|g|\varphi(|g|)|\t_{h_0}|$ where $\l h_0 \r = \C_{\l h \r}(g)$.  
Thus, there are $2\varphi(|h|)/|\t_{h_0}|$ inequivalent rotary pairs for $G$ when $G_2\cong\D_{2^e}$ with $e \geqslant 3$, and $\varphi(|h|)/|\t_{h_0}|$ otherwise.  
This completes the proof of part~(4).

\end{proof}

\section{Groups of types~(II)-(III): 2-nilpotent groups}

In this section, we determine rotary pairs and reversing triples in the 2-nilpotent groups $G$ given in parts~(II)-(III) of Proposition~\ref{lem:G}.
The first lemma treats the groups in (II) of Proposition~\ref{lem:G}.

\begin{lemma}\label{lem:2-nil-2}
Let $G=(\l g\r\times\l g_v\r){:}(\l u,v\r\times\l h\r)$, where $|g|,|g_v|$ and $|h|$ are pairwise coprime, $\C_{\l h\r}(gg_v)\leqslant \Phi(\l h\r)$, and $\l u, v\r=\D_{2^e}$ with $e\geqslant 2$ such that \[(g,g_v)^{v}=(g,g_v^{-1}) \text{ and } (g,g_v)^{u}=(g,g_v),\]
where  $g_v\neq 1$.

Let $\calM$ be a $G$-arc-transitive map.
Assume that $G\leq\Aut(\calM)$.
Then either
\begin{enumerate}[{\rm(1)}] 
	\item $\calM$ is $G$-vertex-reversing, $G=\l g_v\r{:}\l u,v\r$, relabeling if necessary, a reversing triple $(x,y,z)$ is equivalent to 
	\[(u,g_vw^{2j}v,v)\text{ , or }(v,g_vw,u)\text{ when } |w|=2,\] and there are $6\cdot 2^{e-2}$ inequivalent reversing triples of $G$; or\vskip0.1in
	
	\item $\calM$ is $G$-vertex-rotary, and a rotary pair $(a,z)$ is equivalent to  
	\[(gg_vwh^j,v) \text{ or }(gg_vvh^j,v),\]
	where $\gcd(j,|h|)=1$, and there are $2\varphi(|h|)/|\t_{h_0}|$ inequivalent rotary pairs of $G$.
\end{enumerate}
\end{lemma}

\begin{proof}
Note that $\C_{\l u,v\r}(g)=\l\rho^2,u\r=\D_{2^{e-1}}$ since $g^u=g$ and $g^v=g^{-1}$.

{\bf Case 1.} 
Suppose that  $(x,y,z)$ is a reversing triple of $G$.

Firstly, we claim that $h=g=1$, and
\[G= \l g_v\r{:}\l u,v\r=(\l g_v\r\times\l w^2,u \r){:}\l v\r.\]
By definition, $\C_G(g)\geqslant \l g,g_v\r{:}\l u,v\r$, and $|h|$, $|g|$ are odd.
It follows that $h=1$, as $\l h\r\cong G/(\l g\r{:}G_2)$ is generated by involutions, and so $g=1$, as $G=\l g\r\times \l g\r{:}\l u,v\r$. 
The claim is thereby proved.

Next, we determine the triple $(x,y,z)$.
The quotient group $G/\l g\r$ is dihedral, hence is generated by two involutions of $x$, $y$ and $z$, say $x$ and $z$ without loss of generality.
Hence $\l x,y\r$ is a Sylow $2$-subgroup of $G$, and so we may assume that $x,\, z\in \l u,v\r$ by Sylow's theorem.
By Lemma~\ref{lem:G}, $\Aut(G)=\l\tilde{g_v}\r{:}\Aut(g_v)\times \l \xi^2\r{:}\l \mu\r$, where $ \l \xi^2\r{:}\l \mu\r\leqslant\Aut(G)_{g_v}\cap \Aut(G)_{\l u,v\r}$.
Hence $\{x,y\}$ is equivalent to $\{u,v\}$.
Now we have that $y=g_v^{i}t$, where $t\in \l u,v\r$, and the integer $i$ satisfies $\gcd(i,|g_v|)=1$.
Otherwise $\l x,y,z\r\leqslant\l g_v^{i},u,v,w^{2j}v\r<G$.
Also, the involution $t$ is either $w^{2j}v$ or $w$ with $|w|=2$.

Consequently, the triple $(x,y,z)$ is equivalent to
\[(u,g_vw^{2j}v,v)\text{ , or }(v,g_vw,u)\text{ when } |w|=2,\]
permuting $x$, $y$ if necessary.
Therefore, the group $G$ has $2^{e-1}$ inequivalent reversing triples, as in statement~{\rm (1)}.

{\bf Case 2.}
Suppose that  $(a,z)$ be a rotary pair of $G$.
By Sylow's theorem, we may assume that $z\in\l u,v\r$. 
Furthermore, $z\notin\l w^2,u\r$ since $G/\C_{\l u,v\r}(gg_v)$ is not cyclic. 
Hence, we may take 
\[z=v.\]
(Under $\l \mu\r$).
Without loss of generality, there is 
\[a=(gg_v)^ith^j,\]
where $i$, $j$ are integers and $t\in \l u,v\r$.
Integers $i$ and $j$ satisfy $\gcd(i,|gg_v|)=1$ and $\gcd(j,|h|)=1$, respectively, and $\l t,v\r\cong \l u,v\r$, since $G=\l a,z\r\leqslant\l (gg_v)^i,t,v,h^j\r$.
Therefore, $t=w^{2k+1}$ or $w^{2k}u$.
Finally, by Theorem~\ref{lem:G}, pairs $((gg_v)^iw^{2k+1}h^j,v)$ and $((gg_v)^iw^{2k}uh^j,v)$ are equivalent to $(gg_vwh^j,v) $ and $(gg_vvh^j,v)$, respectively.
So there are $2\phi(|h|)/|h_0|$ inequivalent rotary pairs of $G$.
This is as in part~(2).
\end{proof}

Next, we consider the groups listed in (III) of  Proposition~\ref{lem:G}.

\begin{lemma}\label{lem:2-nil-reversing}
Let $G=(\l g_u\r\times\l g_v\r\times\l g\r){:}(\l u,v\r\times\l h\r)$, where $|g_u|,|g_v|,|g|$ and $|h|$ are pairwise coprime, $\C_{\l h\r}(g_ug_vg)\leqslant \Phi(\l h\r)$, and $\l u, v\r=\D_{2^e}$ with $e\geqslant 2$ such that
\[\mbox{$(g_u,g_v, g)^v=(g_u,g_v^{-1},g^{-1})$ and $(g_u,g_v,g)^u=(g_u^{-1},g_v,g^{-1})$,}\]
where if $g_u=1$, then $g_v\not=1$ and $g\not=1$;

Let $\calM$ be a $G$-arc-transitive map.
Assume that $G\leq\Aut(\calM)$.
Then $g_v$ or $g_u=1$.

Suppose that $g_v=1$. 
Then either  
\begin{enumerate}
	\item[{\rm (1)}] $\calM$ is $G$-vertex-reversing, $G=\l g,g_u\r{:}\l u,v\r$, and relabeling if necessary, a reversing triple $(x,y,z)$ is equivalent to 
	\[(g_u g^{k_1} w^{2i} u,\, g^{k_2} v,\, u),\]
	where $\gcd(k_1, k_2) = 1$, $\ell_1$ is odd, and there are
	\[
	 6\cdot 2^{e-3} \cdot \frac{|\{(k_1, k_2) \mid \gcd(k_1, k_2) = 1\}|}{\varphi(|g|)}
	\]
	inequivalent reversing triples of $G$;\vskip0.1in
	
	\item[{\rm (2)}] $\calM$ is $G$-vertex-rotary, and a rotary pair $(a,z)$ is equivalent to
	\[(g_u g w h^\ell, u)\text{ or }(g_u g w^{2} v h^\ell, u),\]
	and there are $2 \varphi(|h|)/|\tau_{h_0}|$ inequivalent rotary pairs in $G$.
\end{enumerate}

\end{lemma}

\begin{proof}

{\bf Case 1.}
Suppose that $(x, y, z)$ is a reversing triple of $G$. 
Then the element $h$ of odd order must satisfy $h = 1$, because the quotient group $\l h \r \cong G / \l g_1, g_2, g_3, u, v \r$ is generated by involutions. 
This implies 
\[ G = (\l g_1 \r \times \l g_2 \r \times \l g_3 \r) {:} \l u, v \r, \]
and we proceed to classify the involutions of $G$.
Without loss of generality, let $t = g_u^i g_v^j g^k w^\ell u$ be an involution for integers $i, j, k, \ell$. 
If $\ell$ is odd, then conjugation by $w^\ell u$ maps $(g_u, g_v, g)$ to $(g_u, g_v^{-1}, g^{-1})$. Therefore,
\[ t^2 = g_u^i g_v^j g^k w^\ell u g_u^i g_v^j g^k w^\ell u = g_u^i g_v^j g^k \cdot g_u^{i} g_v^{-j} g^{-k} = g_u^{2i} = 1. \]
Since $|g_u|$ is odd, this forces $g_u^i = 1$. 
Thus, $t = g_v^j g^k w^\ell u$. 
Similarly, if $\ell$ is even, we derive $g_v^j = 1$ and $t = g_u^i g^k w^\ell u$. 
Additionally, the involution $w^{2^{e-2}}$ cannot belong to any reversing triple.

By Sylow’s theorem, we may assume $z \in \l u, v \r$, so 
\[ z = w^m u \]
for some integer $m$. 
Let 
\[ x = g_u^{i_1} g_v^{j_1} g^{k_1} w^{\ell_1} u,\  y = g_u^{i_2} g_v^{j_2} g^{k_2} w^{\ell_2} u, \]
where $i_1, j_1, k_1, \ell_1, i_2, j_2, k_2, \ell_2$ are integers. 
If both $\ell_1$ and $\ell_2$ are even, then $g_v^{j_1} = g_v^{j_2} = 1$. However, this leads to $\l x, y, z \r \leqslant \l g_u, g, u, v \r < G$, a contradiction. 
Similarly, both $\ell_1$ and $\ell_2$ being odd also yields a contradiction. 
Hence, after relabeling, we may take $\ell_1$ even and $\ell_2$ odd. 
This implies $g_v^{j_1} = 1$ and $g_u^{i_2} = 1$, simplifying to: 
\[ x = g_u^{i_1} g^{k_1} w^{\ell_1} u, \quad y = g_v^{j_2} g^{k_2} w^{\ell_2} u. \]
Since $\l x, y, z \r \leqslant \l g_u^{i_1}, g^{k_1}, w^{\ell_1}, u, g_v^{j_2}, g^{k_2}, w^{\ell_2} u, w^m u \r = G$, we require 
\[ \gcd(i_1, |g_u|) = \gcd(j_2, |g_v|) = \gcd(k_1, k_2) = 1. \]
By Theorem~\ref{lem:G}, 
\[ \Aut(G)_{(L)} \geqslant \Aut(H)  \text{ and }  \Aut(G)_L \cap \Aut(G)_{(H)} \geqslant \l \xi^2 \r {:} \l \mu \r. \]
Consequently, the triple $(x, y, z)$ is equivalent to either 
\[ (g_u g^{k_1} w^{\ell_1} u,  g_v g^{k_2} v, u) \text{ or } (g_u g^{k_1} w^{\ell_1} u,  g_v g^{k_2} v,  v), \]
as stated in part~(i). 

Now we verify whether the above triples generate $G$. 
First, consider the triple $(x, y, z) = (g_u g^{k_1} w^{\ell_1} u,\, g_v g^{k_2} v,\, u)$, where $\ell_1$ is even.
Using the conjugation relation $(g_u, g_v, g)^w = (g_u^{-1}, g_v^{-1}, g)$ and the condition $\gcd(k_1, k_2) = 1$, we compute:
\[
\begin{aligned}
	\l x, y, z \r &= \l g^{k_1},\, g^{k_2},\, g_u,\, w^{\ell_1},\, g_v w^{-1},\, u \r \\
	&= \l g, g_u,\, w^{\ell_1},\, g_v w^{-1},\, u \r.
\end{aligned}
\]
Observe that the quotient group
\[ G=\l x, y, z \r / \l g, g_u \r \cong \l w^{\ell_1},\, g_u w^{-1},\, u \r \cong \l u, v \r\] forces $g_v = 1$. 
A similar argument shows $g_u = 1$ for the cases that  $(x, y, z)$ is equivalent to $(g_u g^{k_1} w^{\ell_1} u,\, g_v g^{k_2} v,\, v)$. 

Conversely, if $g_v = 1$ (respectively $g_u = 1$), the triple $(g_ug^{k_1} w^{\ell_1} u,\,  g^{k_2} v,\, u)$ (respectively $(g_u g^{k_1} w^{\ell_1} u,\, g^{k_2} v,\, v)$) indeed generates $G$, by the argument above.

Thus, $G$ has a reversing triple $(x, y, z)$ if and only if one of the following holds:
\begin{enumerate}
	\item $g_v = 1$, and the triple is equivalent to $(g_u g^{k_1} w^{\ell_1} u,\, g^{k_2} v,\, u)$; \vskip0.1in
	\item $g_u = 1$, and the triple is equivalent to $(g^{k_1} w^{\ell_1} u,\, g_v g^{k_2} v,\, v)$,
\end{enumerate}
where $\gcd(k_1, k_2) = 1$, $\ell_1=2i$ is even.

The number of inequivalent reversing triples is given by:
\[
\frac{|\{(k_1, k_2) \mid \gcd(k_1, k_2) = 1\}|}{\varphi(|g|)} \cdot 2^{e-2}.
\]

{\bf Case 2.}
Suppose that $(a,z)$ is a rotary pair of $G$. 
By Sylow's theorem, we may assume that $z \in \langle u,v \rangle$, so that
\[ z = w^m u, \]
for some integer $m$. 
Without loss of generality, 
\[ a = g_u^i g_v^j g^k t h^{\ell}, \]
where $t \in \langle u, v \rangle$, and $i, j, k, \ell$ are integers. 
These integers $i, j, k, \ell$ and the element $t$ satisfy  
\[ \gcd(i, |g_u|) = \gcd(j, |g_v|) = \gcd(k, |g|) = \gcd(\ell, |h|) = 1 \ \text{and}\ \langle t, \rho^m u \rangle = \langle u, v \rangle, \]
so that $\langle a, z \rangle \leqslant \langle g_u^i, g_v^j, g^k, t, w^m u, h^{\ell} \rangle = G$. 
Here by Theorem~\ref{lem:G}, 
\[\Aut(G)_{(L)}\geqslant \Aut(H)\text{ and }\Aut(G)_{L} \cap \Aut(G)_{(H)}\geqslant \langle \xi^2 \rangle {:} \langle \mu \rangle, \]
where $H=\l g_ug_vg\r$ and $L=\l u,v\r\times \l h\r$.
The pair $(a, z)$ is equivalent to 
\[ (g_u g_v g t h^{\ell}, z), \]
where $\gcd(\ell, |h|) = 1$ and $(t,z)\in \{(w^i,u),(w^i,v),(w^{2i}v,u),(w^{2i}u,v)\}$ with $i$ being an odd integer. 

Next, we verify whether these pairs generate $G$. 
The key observation is that $G' = \l [a, z] \r$, which holds because 
\[ \l [a, z] \r \char \l g_u g_v g \r \times \l w^2 \r \lhd G. \]

Suppose $(a, z) = (g_u g_v g w^i h^\ell, u)$ with $i$ odd. 
Then $G' = \l [a, z] \r = \l g_u^2 g^2 w^2 \r$, and the abelian quotient 
\[ G/G' \cong \l g_v \r \times \l u, v \r / \l w^2 \r \times \l h \r. \]
This implies $g_v^v = g_v$ since $\gcd(|G'|, |g_v|) = 1$. 
From the definition of $g_v$, it follows that $g_v = 1$. 
Similarly, if $z = v$, then $g_u = 1$ must hold.

Conversely, if $g_u = 1$ (respectively $g_v = 1$), the pairs $(g_v g w^i h^\ell, u)$ and $(g_v g w^{2i} v h^\ell, u)$ (respectively $(g_u g w^i h^\ell, v)$ and $(g_u g w^{2i} u h^\ell, v)$) generate $G$ by the argument above.

In summary, $G$ admits a rotary pair $(a, z)$ if and only if one of the following holds:
\begin{enumerate}
	\item $g_v = 1$, and $(a, z)$ is equivalent to $(g_u g w h^\ell, u)$ or $(g_u g w^2 v h^\ell, u)$;
	\vskip0.1in
	\item $g_u = 1$, and $(a, z)$ is equivalent to $(g_v g w h^\ell, v)$ or $(g_v g w^{2} u h^\ell, v)$;
\end{enumerate}
Furthermore, there are $\frac{2 \varphi(|h|)}{|\tau_{h_0}|}$ inequivalent rotary pairs in $G$.
\end{proof}

\section{Groups of types~(IV)-(V)}

In this section, we assume that $G$ is a group of type (IV) or type (V).

We analyse the two cases separately.
\subsection{Groups of type {\rm (IV)}}
The first lemma treats groups of type~IV.
\begin{lemma}\label{lem:not-2-nil-1}
Let $G=(\l g\r\times \l u,v\r){:}\l h\r=(\ZZ_m\times\ZZ_2^2){:}\ZZ_{3^fn}$ with $G/\O_{2'}(G)=\A_4$, where $m,n,6$ are pairwise coprime.
Then $\calM$ is $G$-vertex-rotary with a rotary pair $(a,z)$, and
\begin{enumerate}[{\rm(1)}]
\item $G=\l u,v\r{:}\l h\r=\ZZ_2^2{:}\ZZ_{3\lambda}=\ZZ_\lambda.\A_4$, and\vskip0.1in
\item $a=h^i$ with $\gcd(i,3\lambda)=1$, and $z\in\{u,v,uv\}$, and\vskip0.1in
\item the underlying graph of $\calM$ is $\Ga=\K_4^{(\lambda)}$.
\end{enumerate}
\end{lemma}
\begin{proof}
Let $N=\l u,v\r=\ZZ_2^2$, and let $C=\C_G(N)$.
Then $G/C\leq\Aut(N)$, and hence $G/C=\ZZ_3$.
It follows that $\calM$ is $G$-vertex-rotary, and $\l g\r{:}\l h\r\cong G/N$ is of odd order and generated by a single element.
Thus $g=1$, and $G=N{:}G_{2'}=\l u,v\r{:}\l h\r\cong\ZZ_\lambda.\A_4$, with $\lambda=|h|/3$, as in part~(1).

Let $(a,z)$ be a rotary pair for $G$.
Then $z\in\{u,v,uv\}$, and $a$ is conjugate to a generator of $\l h\r$.
Without loss of generality, we may let $a=h^i$, with $\gcd(i,|h|)=1$, as in part~(2).

Finally, the vertex stabilizer $G_\a=\l a\r\cong\ZZ_{3\lambda}$, and $|G:G_\a|=4$.
Thus $\lv V \rv=4$, and the kernel $K=G_{(V)}=\ZZ_\lambda$, and so $\Ga=\K_4^{(\lambda)}$, as in part~(3).
\end{proof}

\subsection{Groups of type {\rm (V)}}
%
\begin{lemma}\label{lem:set up}
	Let $G=(\l g\r\times\l w^2,v\r){:}((\l h\r{:}\l u\r)\times\l c\r)$, where $\l u,v\r=\D_8$, $\l h\r$ is a Sylow $3$-subgroup, both $\l g\r$ and $\l c\r$ are Hall subgroups such that $\C_{\l c\r}(g)\leqslant\Phi(\l c\r)$, $(w^2v,v,w^2)^h=(v,w^2,w^2v)$ and $h^u=h^{-1}$.
	
	Assume that $G$ is indecomposable, and $G$ has a rotary pair or a reversing triple.
	Then $g^{u}=g^{-1}$ and $[g,h]=1$.
\end{lemma}
\begin{proof}

Suppose that $\l g_p\r$ be a Sylow $p$-subgroup of $\l g\r$ such that $[g_p,u]=1$.
Then $\l g_p\r{:}\l c\r$ is non-abelian, otherwise, $G=\l g_p\r\times G_{p'}$, contradiction $G$ being indecomposable.
Thus the odd order  factor group 
\[\overline{G}:=G/\l g_{p'},u,v,h\r\cong \l g_p\r{:}\l c\r\]
is non-abelian.

Let $(a,z)$ be a rotary pair and $(x,y,z)$ be a reversing triple, and so either $G=\l a,z\r$ or $G=\l x,y,z\r$.
By Sylow's theorem, we assume that $z\in \l u,v\r$.
It follows that the factor group $\overline{G}$ is either cyclic or dihedral, contradiction.
Therefore, $g_p^w=g_p^{-1}$ for each $p\in \pi(|g|)$, and so $g^{u}=g^{-1}$.
Let $C=\C_G(g)$.
The abelian group 
\[G/C\geqslant \l uC \r.\]
Hence $h\in C$, equivalently $[g,h]=1$, otherwise $\l hC,uC\r$ is a non-abelian subgroup of $G/C$.
\end{proof}
Before treating the groups of type~(V), we give a typical construction for rotary pairs in this case, which is built upon the symmetric group $G=\Sym\{1,2,3,4\}=\S_4$, and a rotary pair $((123),(14))$.

\begin{example}\label{cons:Rota-(V)}
{\rm
We construct rotary pairs in some groups of type~(V).
\begin{enumerate}[{\rm (1)}]
\item Let $G=\Sym\{1,2,3,4\}=\l u,v\r{:}(\l h_0\r{:}\l w\r)=\ZZ_2^2{:}\S_3=\S_4$, where $u=(12)(34)$, $v=(13)(24)$, $w=(14)$ and $h_0=(124)$.
Let
\[(a_1,z_1)=((123),(14)).\]

\item Let $X=\l u,v\r{:}(\l h\r{:}\l w\r)=\ZZ_2^2{:}\D_{2.3^f}=\l h^3\r.G=\ZZ_{3^{f-1}}.\S_4$, where $\ZZ_2^2=\l u,v\r<\l u,w\r=\l v,w\r=\D_8$, $|h|=3^f$, and $(u,v,uv)^h=(v,uv,u)$, and $h^w=h^{-1}$.
Let
\[\mbox{$(a_2,z_2)=(h,w^u)$.}\]
Noticing that $h^u=uhu=hvu$ and $(h^u)^w(h^u)^{-1}=h^{-2}$, we have that $h\in\l h^u,w\r$, and hence $\l h^u,w\r=X$.
So $(a_2,z_2)=(h^u,w)^u$ is a rotary pair for $X$.

\item Let $Y=\l g\r{:}X=(\l g\r\times \l u,v\r){:}(\l h\r{:}\l w\r)=(\ZZ_m\times\ZZ_2^2){:}\D_{2.3^f}$, where $|g|=m$ is coprime to 3, $gh=hg$ and $g^w=g^{-1}$.
    Let
    \[\mbox{$(a_3,z_3)=(ga_2,z_2)=(gh,w^u)$.}\]
    Then $\l a_3,z_3\r=\l gh,z_2\r=\l g,h,z_2\r=\l g\r{:}\l h,z_2\r=Y$, so that $(a_3,z_3)$ is a rotary pair for the group $Y$.

\item Let $Z=Y{:}\l c\r=(\l g\r\times \l u,v\r){:}((\l h\r{:}\l w\r)\times \l c\r)=(\ZZ_m\times\ZZ_2^2){:}(\D_{2.3^f}\times\ZZ_n)$, where $[h,Y]=1$ and $[\l g\r,\l c\r]\le\Phi(\l g\r)$.
    Let
    \[\mbox{$(a_4,z_4)=(a_3c,z_3)=(ghc,w^u)$.}\]
    Then $\l a_4,z_4\r=\l c,gh,z_3\r=Y{:}\l c\r=Z$, so $(a_4,z_4)$ is a rotary pair for $Z$.
    \qed
\end{enumerate}
}
\end{example}

The following lemma determines rotary pairs and reversing triples in the groups of type~(V).
Let
\[\mbox{$G=(\l g\r\times\l w^2,v\r){:}((\l h\r{:}\l u\r)\times\l c\r)$,}\]
a group of type~(V).
We first consider the subgroup $\l w^2,v\r{:}\l h,u\r$ of $G$.

\begin{lemma}\label{th-w}
	Following notations above,  $t\in\l w^2,v\r$ is such that $\l t h,u\r=\l w^2,v\r{:}\l h,u\r$ if and only if $t=v$ or $w^2$.
\end{lemma}

\begin{proof}
	From $(w^2v,v,w^2)^h=(v,w^2,w^2v)$ and $h^u=h^{-1}$, we have:
	\[
	\l t h,u\r=\D_{4\cdot3^e} \Longleftrightarrow (th)^u=(th)^{-1} \Longleftrightarrow t^uh^{-1}=h^{-1}t \Longleftrightarrow t^u=t^h \Longleftrightarrow t\in\{1,w^2v\}.
	\]
	Since $\D_{4\cdot3^e}$ is a maximal subgroup of $\l w^2,v\r{:}\l h,w\r\cong \ZZ_{3^{e-1}}.\S_4$, we have that $\l t h,u\r=\l w^2,v\r{:}\l h,w\r$ if and only if
	\[
	t \in \l u,v\r {\setminus}\{1,w^2v\}=\{v,w^2\}.
	\]
	
	Suppose that $t=w^2$.
	Then $(w^2h)^2u^{w^2h}=w^2vh^{2}h^{-2}u=w$, and so $\l w^2h,u\r=\l w^2,v\r{:}\l h,u\r$.
	Similarly, it is straightforward to check that $\l vh,u\r=\l w^2,v\r{:}\l h,u\r$.
\end{proof}

Now, we are ready to classify rotary pairs and reversing triples of groups of type V.

\begin{lemma}\label{lem:Rev-(V)}
	Following notations in Lemma~\ref{lem:set up}.
	Assume that $\calM$ is $G$-vertex-reversing with a reversing triple $(x,y,z)$, and 
	let $t$, $t'\in \l u,v\r$, $i$ be an integer such that $\gcd(i,|gh|)=1$. 
	Then $c=1$, and relabeling if necessary, 
	the reversing triple $(x,y,z)$ is equivalent to one of the following triples:
	\begin{enumerate}
		\item $(t (gh)^i u, t', u)$, where $t'\neq 1$;\vskip0.1in
		\item $(t (gh)^i u, g^{i'} h^{3j_0} w^2u, u)$;\vskip0.1in
		\item $(t (gh)^i u, g^{i'} t' h^{j'} u, u)$, where $(t, t')\neq (1, 1)$
	\end{enumerate}
\end{lemma}
\begin{proof}
	First, $c=1$, since every factor group of $G$ is generated by involutions.
	By Lemma~\ref{lem:set up}, we have that $[g,h]=1$.
	Thus $\overline{G}:= G / \l w^2, v \r \cong \l g h \r {:} \l u \r$ is a dihedral group.
	
	For each $l \in G$, let $\overline{l}$ denote the image of $l$ in $\overline{G}$.  
	Among the three involutions $\overline{x}$, $\overline{y}$, and $\overline{z}$, two generate the dihedral group $\overline{G}$. Without loss of generality, assume these are $\overline{x}$ and $\overline{z}$.  
	It follows that $z \notin \l w^2, v \r$, and therefore we may assume  
	\[ z = u \]  
	by Sylow's theorem.  
	Write  
	\[ x = t(gh)^i u \]  
	where $t \in \l w^2, v \r$ and $\gcd(i, |gh|) = 1$.  
	Since $x$ is an involution, the element $t$ satisfies  
	\[ t^u = t^{h^i}, \]  
	implying $t \in \{w^2v, 1\}$ if $i \equiv 1 \pmod{3}$, or $t \in \{v, 1\}$ if $i \equiv 2 \pmod{3}$.  
	
	The remaining involution $y$ is either  in $\l w^2,v\r$ or  of the following:  
\[
\begin{aligned}
	&\ g^i w^2 v h^j u = g^i h^j v u && (j \equiv 1 \pmod{3}); \\
	&\ g^i v h^j u = g^i h^j w       && (j \equiv 2 \pmod{3}); \\
	&\ g^i h^{3j_0} w^2 u\text{ or }g^i h^j u.
\end{aligned}
\]	

	 Suppose $y \in \{w^2v, v, w^2\}$.  
	Then $\{w^2v, v, w^2\} = y^{\l xz \r}$, and thus $\l x, y, z \r = \l t(gh)^i u, t', u \r = G$, as in case~(1).  
	
	 Suppose $y = g^{i'} h^{3j_0} w^2 u$.  
	Then $\l y, z \r = (\l w^2 \r \times \l g^{i}, h^{3j_0} \r) {:} \l u \r$, which implies $w^2 \in \l x, y, z \r$.  
	Since $\{w^2v, v, w^2\} = (w^2)^{\l xz \r}$, it follows that $\l x, y, z \r = G$, as in case~(2).  
	
	 For the remaining choices of $y$, direct verification shows $\l x, y, z \r = G$ except when $t = 1$ and $y = (gh)^{i'} u$, as described in case~(3).

%
%
%
%
%
%

\end{proof}

\begin{lemma}\label{lem:Rota-(V)}
	Following notations in Lemma~\ref{lem:set up}.
	Assume that $\calM$ is $G$-vertex-rotary with a rotary pair $(a,z)$.
	Then  $(a,z)$ is equivalent to 
	\[(g t h^j c^k, u)\text{ or }(g t h^j u c^k, u),\] 
	where  $\gcd(j,3)=\gcd(k,|c|)=1$, and $t\in \{v, w^2\}$.
	There are 
	\[\varphi(|h|) \varphi(|c|) / 2 |\t_{h^3}| |\t_{c_0}|\]
	inequivalent rotary pairs of $G$.
	
\end{lemma}

\begin{proof}
	By Sylow's theorem, we may assume that the involution $z\in \l u,v\r$.
	Suppose that $z\in \l w^2,v\r$.
	Then $G/\l u,v\r\cong \l g\r{:}((\l h\r{:}\l u\r)\times\l c\r)$ is isomorphic to a subgroup of $\l a\r$.
	This is not possible since $\l g\r{:}((\l h\r{:}\l u\r)\times\l c\r)$ is non-abelian.
	Hence $z\in\l u,v\r\setminus\l w^2,v\r=\{u,w^2u\}$, and we may further let
	\[z=u,\]
	as $(w^2u)^{v}=w$.
	
	Since $G=(\l g\r\times\l w^2,v\r){:}((\l h\r{:}\l u\r)\times\l c\r)$, the element $a$ has two possibilities:
	\[\mbox{either $a=g^i t h^j u c^k$, or $a=g^i t h^j c^k$,}\]
	where $i$, $j$, $k$ are integers and $t\in \l w^2,v\r$.
	In both cases, one always have that 
	\[G=\l a,z\r\leqslant \l g^i, t, h^j, w, c^k\r,\]
	which implies that $\gcd(i,|g|)=\gcd(j,3)=\gcd(k,|c|)=1$, and $t\not=1$.
	Again in either cases, one have that $\l w^2,v\r{:}\l h,u\r\times \l c\r \cong\l a,z\r/\l g\r\cong \l th^jc^k,u\r$. Hence $t=v$ or $w^2$ by Lemma~\ref{th-w}.
	
	Conversely, we show explicitly that the above pairs generate $G$.
	Without loss of generality, let $z=u$ and $a=gthu^{\ell}c$, where $t\in \{v,w^2\}$ and $\ell\in \{0,1\}$.
	Then $\l a,z\r=\l gthc, u\r$, and we let $a'=gthc$.
	Then
	\[{a'}^2=g g^{c^{-1}} t t^{h^{-1}} h^2 c^2\text{ and }z^{a'}z= (g^2)^{c h^{-1}} (t^u t)^{h^{-1}} h^2.
	\] 
	Since $\gcd( |g|, 6 )=1 $ and $[g,\l w^2,v\r{:}\l h\r]=1$, we have that 
	\[g\in \l z^{a'}z\r.\]
	So $\l t t^{h^{-1}} h^2 c^2\r\leqslant \l g,{a'}^2\r$.
	As $\gcd(|c|, 6)=1$ and $[c, \l w^2,v\r{:}\l h\r ] = 1 $, 
	\[ c\in \l g,{a'}^2\r .\]
	It follows from Lemma~\ref{th-w} that $\l a,z\r=\l g, c, th, u\r=G$ as $t\in \{w^2,v\}$.
	
\end{proof}

\section{Proof of Theorem~\ref{thm:main}}

Let $\calM=(V,E,F)$ be a $G$-arc-transitive map with $G\le\Aut(\calM)$ and underlying graph $\Ga=(V,E)$.
Assume that $\gcd(\chi(\calM),|E|)=1$.
Then each Sylow subgroup of $G$ is cyclic or dihedral by Lemma~\ref{(chi,|A|)=1}, and then $G$ is divided into five types in Proposition~\ref{lem:G}.
The groups of type~(I) are treated in Lemma~\ref{lem:G-dih}-\ref{lem:G-metac}, which show that $\calM$ and $G$ satisfy part~(1) of Theorem~\ref{thm:main}.
Then the groups of types~(II)-(III) are dealt with by Lemmas~\ref{lem:2-nil-2} and \ref{lem:2-nil-reversing}, which show that $\calM$ and $G$ satisfy parts~(2)-(3) of Theorem~\ref{thm:main}.
Finally, the groups of types~(IV)-(V) are treated in Lemmas~\ref{lem:not-2-nil-1}, \ref{lem:Rota-(V)} and \ref{lem:Rev-(V)}.
This completes the proof of Theorem~\ref{thm:main}.
\qed

\end{document}